\documentclass[11pt]{amsart}
\usepackage{geometry}                
\usepackage{graphicx}
\usepackage{epstopdf}
\usepackage[usenames]{color}

\usepackage{amssymb, amsthm, comment, color}
\usepackage[labelfont=small]{caption, subcaption}

\newtheorem{assumption}{Assumption}
\newtheorem{thm}{Theorem}

\newtheorem{lemma}{Lemma}

\theoremstyle{remark}
\newtheorem*{rem}{Remark}

\title{Knotted and linked products of recombination on $T(2,n)\#T(2,m)$ substrates}
\author{Erica Flapan}
\address{Department of Mathematics, Pomona College, Claremont, CA 91711, USA}

\author{Jeremy Grevet, Qi Li, Chen Sun, Helen Wong}
\address{Department of Mathematics, Carleton College, Northfield, MN 55057, USA}


\thanks{The first author was supported in part by NSF Grant DMS-0905087, the last author was supported in part by NSF Grant DMS-1105692. }
\begin{document}
\maketitle

\begin{abstract}  We develop a topological model of site-specific recombination that applies to substrates which are the connected sum of two torus links of the form $T(2,n)\#T(2,m)$.  Then we use our model to prove that all knots and links that can be produced by site-specific recombination on such substrates are contained in one of two families, which we illustrate.  
\end{abstract}

\section{Introduction}

Knots and links can occur in the axis of DNA molecules as the result of replication, recombination, and enzyme actions.  Such knots and links have been used to better understand the mechanism of how topoisomerase enzymes change the topology of DNA during site-specific recombination.  The difficulty in understanding this mechanism is that there is no way to observe the details of an enzyme action while it's taking place, and hence the mechanism cannot be studied directly.  
To solve this problem, molecular biologists allow a topoisomerase enzyme to act on a closed circular DNA {\it substrate}.    They then use electron microscopy or gel electrophoresis to try to identify the knotted and linked {\it products} of recombination mediated by the enzyme, and finally use topology to try to deduce the mechanism of the enzyme that could have produced such products from the known knot or link type of the substrate.  However, neither identifying the knots and links nor using topology to explain the mechanism is easy.  

Major progress on the topological aspect of this problem was made when Ernst and Sumners \cite{ErnstSumners} developed the {\it tangle model of recombination} to explain how the enzyme Tn3 resolvase acted processively on an unknotted substrate to sequentially produce a Hopf link, then a figure eight knot, then a Whitehead link, and finally a $6_2$ knot.  Since then, the tangle model has been further developed and used to explain the actions of a number of other enzymes (see for example \cite{BuckMauricio, BuckMarcotteTn3,  Cabrera, SumnersCozzarelliPhageLambda,   DarcyEtAlXer, DarcyLueckeVazquez,ErnstSumners99, Kauffman, Kim4string, Sumners, VazquezSumnersGin}).  For enzymes whose knotted and linked products have been identified, this approach has been quite successful at explaining the mechanism of the enzyme action.  

However, identifying the knotted and linked products of a enzyme action has its own challenges.  In order to visualize the  products using electron microscopy, molecular biologists need to have a large quantity of knotted and linked products, which can be quite difficult to obtain.  The products are then coated with a substance known as RecA to make them visible with an electron microscope. But even so, it is often hard to see which strand goes over and which strand goes under at each crossing.  If even one crossing is misclassified, the knot or link may be incorrectly identified.  Using gel electrophoresis rather than electron microscopy has the advantage that it does not require a large number of knotted or linked products.  However, gel electrophoresis does not enable scientists to visualize the products.  Rather, it separates the knotted and linked  molecules according to their compactness, which is then correlated with their minimal crossing numbers.  Unfortunately, in general, gel electrophoresis cannot completely identify the particular knots or links from among all those in the tables with a given minimal crossing number.

In contrast with the tangle model which was developed to explain the mechanism of particular enzyme actions, Buck and Flapan \cite{BuckFlapan} developed a model to make it easier for molecular biologists to identify the knotted or linked products of recombination.  In particular, they showed that all knots and links that could be produced by site-specifc recombination with substrates that were (possibly trivial) $T(2,n)$ torus links were in a single family (illustrated in Figure~\ref{figure:BFknots}).  Knowing that only this family of knots and links could be produced from such substrates, helps molecular biologists to identify products of site specific recombination on these substrates based solely on the information they obtain from gel electrophoresis.  The approach introduced in \cite{BuckFlapan} was extended to apply to substrates that are twist knots in \cite{BuckValencia}.

\begin{figure}[htpb]
\includegraphics[width=1.2in]{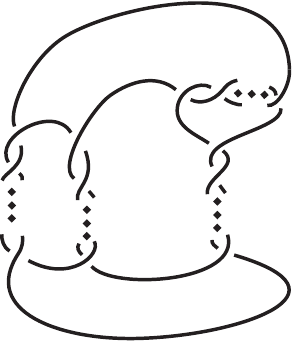}
\caption{The family of products predicted for $T(2,n)$ substrates.} \label{figure:BFknots}
\end{figure}

 In the current paper, we seek to clarify the model developed in \cite{BuckFlapan,BuckValencia}, and extend the results further to apply to substrates that are non-trivial connected sums of torus links of the form $T(2,n)$ and $T(2,m)$ (see Figure~\ref{figure:substrate}).

 \begin{figure}[htpb] 
 \includegraphics[]{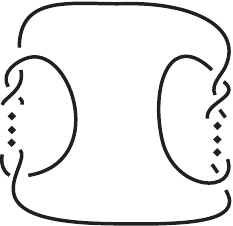}
\caption{We consider substrates of the form $T(2,m) \# T(2,n)$}
 \label{figure:substrate}
\end{figure}

To date, all known knotted and linked products of site-specific recombination on closed circular DNA substrates mediated by serine or tyrosine recombinases are torus links of the form $T(2,n)$, twist knots, and connected sums of the form $T(2,n)\#T(2,m)$ (see \cite{BuckFlapan} for a list of such products).  Such knotted and linked products can be used by molecular biologists as substrates for future experiments.  By characterizing all knots and links that can be produced by site-specific recombination on these substrates, we make the task of identifying the future products of such experiments easier for molecular biologists.

\section{The assumptions of our model}

 Our model relies on three basic assumptions, which taken together correspond to the three assumptions in \cite{BuckFlapan,BuckValencia}.  However, we have tried to state our assumptions in a more explicit and self-contained way than in \cite{BuckFlapan,BuckValencia}, and hence our wording is somewhat different.  The biological justifications for the three assumptions are given in \cite{BuckFlapanJMB, BuckValenciaJMB}.  

The first two assumptions concern how the topoisomerase enzyme and the substrate sit together in space.  In particular, the first assumption is about the configuration of the substrate near the enzyme, and the second assumption is about the configuration of the substrate away from the enzyme.  The third assumption describes the possible ways that a serine or tyrosine enzyme can alter the strands of the DNA.

\subsection{Configuration of the substrate near the enzyme}

\begin{assumption} \label{assumption:1}  The substrate-enzyme complex can be modeled by a ball $B$ intersecting the substrate $J$ in two short arcs.  Furthermore, $B$ has a properly embedded disk $D$ containing the four points of $J\cap \partial B$ such that some projection of $J \cap B$ onto $D$ has at most one crossing between the two arcs of $J$ and no crossings within a single arc of $J$. \end{assumption}

We now assume that Assumption \ref{assumption:1} holds for our substrate-enzyme complex, and hence the projection of $J \cap B$ onto the disk $D$ has one of the forms illustrated in Figure \ref{figure:assumption1}.   We now fix the ball $B$ and the disk $D$ so that we can refer to these sets throughout the remainder of the paper.  Furthermore, we shall refer to $B\cup J$ as the {\it substrate-enzyme complex} and to $D$ as the {\it projection disk}.

\begin{figure}[htpb] 
 \includegraphics[]{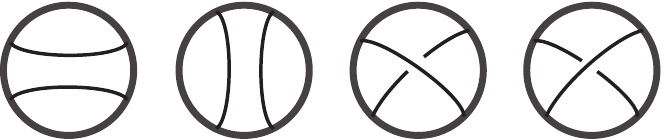}
\caption{Possible projections of $J \cap B$ onto $D$.} \label{figure:assumption1}
\end{figure}

\subsection{Configuration of the substrate away from the enzyme}
Before we can state our second assumption, we need to introduce some terminology.  Consider a surface with boundary $S$ embedded in a sphere $P\subseteq \mathbb{R}^3$.  By a slight abuse of language we will refer to the surface $S$ as a {\it planar} surface.  

Now suppose that the planar surface $S$ is decorated with finitely many arcs $\alpha_i$ whose boundaries are in $ \partial S$.  For each $i$, let $N(\alpha_i)$ denote a tubular neighborhood of $\alpha_i$ in $\mathbb{R}^3$ such that $N(\alpha_i)$ meets $S$ in a neighborhood of $\alpha_i$.  We say that the decorated surface $S$  {\it represents} a surface $R$ in $\mathbb{R}^3$, if $R$ can be obtained from $S$ by replacing each $N(\alpha_i) \cap S$ by a half-twisted band contained in $N(\alpha_i)$.   We may further annotate the decorating arcs with $+$ or $-$ signs to distinguish between the two direction of the twisting.  However, we will generally suppress such designations for the sake of simplicity.  Note that by suppressing this information, a given decorated surface $S$ can represent more than one surface $R$. Figure~\ref{decorated} illustrates an example of a decorated planar surface $S$ representing a surface $R$ in $\mathbb{R}^3$.

\begin{figure}[htpb]
\includegraphics[width=.6\textwidth]{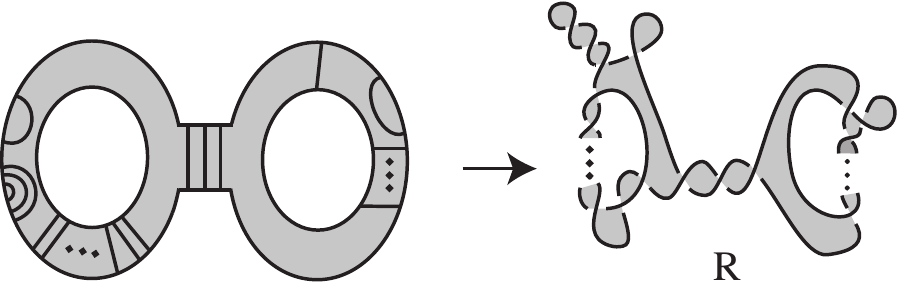}
\caption{The decorated planar $S$ represents the surface $R$.}
\label{decorated}
\end{figure}

In Assumption~2, we will refer to the number of {\it groups of parallel arcs} on a decorated planar surface $S$.  For example, the decorated surface in Figure~\ref{decorated} has seven groups of parallel arcs.

\begin{assumption} \label{assumption:2}  Let $J$ be a connected sum of non-trivial torus links $T(2,n)$ and $T(2,m)$, and suppose that $J\cup B$ is a substrate-enzyme complex with projection disk $D$ satisfying Assumption~1.  Then, after an ambient isotopy of $J$ pointwise fixing the ball $ B$, there is a surface $R$ with boundary $J$ satisfying the following conditions:
\begin{enumerate}

\item $R$ is represented by a decorated 2-holed disk $S$ contained in a sphere $P$ such that $P\cap B=D$;

\item $ \partial B$ is disjoint from the neighborhoods $N(\alpha_i)$ of the decorating arcs of $S$;   

\item $S$ has a minimum number of decorating arcs and a minimum number of groups of parallel arcs among all decorated planar surfaces $S$ satisfying the above conditions.
\end{enumerate}
\end{assumption}

Observe that since $J=T(2,n)\#T(2,m)$, even without Assumption~2 we would know that $J$ bounds a Seifert surface which is represented by a decorated 2-holed disk contained in a sphere.  However, conditions (1) and (2) of Assumption 2 are making the stronger statement that the surface $R$ can be chosen to have a nice relationship to the substrate-enzyme complex $B\cup J$.  In particular, these conditions prevent $B\cup R$ from having one of the forms illustrated in Figure~\ref{figure:assumption2disallowed}.  Biological arguments are given in \cite{BuckFlapanJMB, BuckValenciaJMB} explaining why these types of configurations are unlikely to occur.

\begin{figure}[htpb]
\includegraphics{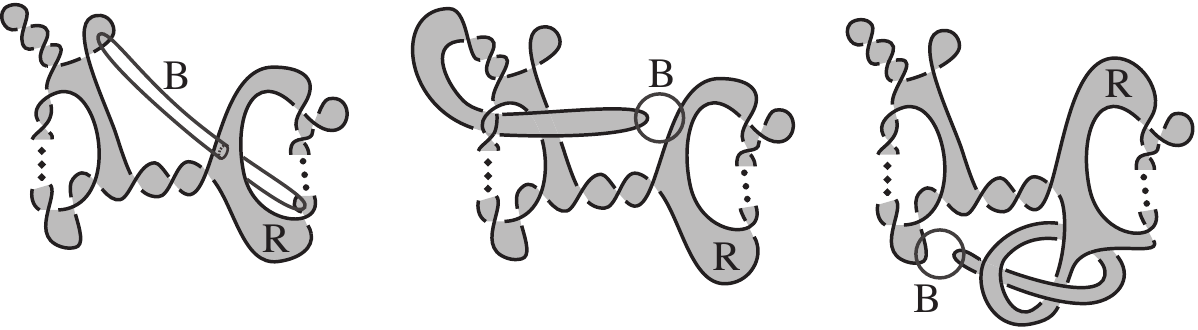}
\caption{Configurations of $R \cup B$ that are disallowed by Assumption 2} \label{figure:assumption2disallowed}
\end{figure}

Condition (3) of Assumption~2 is included so that the decorating arcs on the surface $S$ are as simple as possible.   In particular, condition (3) does not add any additional restrictions on the topology of the substrate-enzyme complex $J\cup B$.

We now assume that Assumption~2 holds for our substrate-enzyme complex, and fix the surface $R$ and sphere $P$, as well as the representation of $R$ by the decorated planar surface $S\subseteq P$ satisfying the conditions in Assumption~2.

\subsection{Possible ways the enzyme can alter the strands}

To further simplify our drawings, we will depict a row of $k$ half-twists between two strands of a link as a rectangle where crossings go between the two short sides of the rectangle, as in Figure \ref{figure:rectangle}.  Positive and negative values of $k$ correspond to positive and negative half-twists between the two strands respectively.  In situations where $k$ can be any integer, we generally omit it from our pictures.

\begin{figure}[htpb]
\includegraphics[height=.8in]{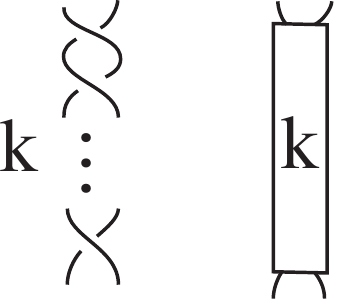}
\caption{A row of half-twists is depicted as a rectangle.} \label{figure:rectangle}
\end{figure}

Our third assumption concerns the serine and tyrosine families of recombinases.  To describe their action on the substrate, 
we use the projection of $J \cap B$ onto the disk $D \subseteq B$ which was fixed in Assumption~1.

\begin{assumption} \label{assumption:3}
The recombinase only alters the substrate $J$ within the ball $B$, changing the topological configuration of the two arcs in $J \cap B$ in the following ways.
\begin{enumerate}
\item A serine recombinase cuts the arcs of $J \cap B$, adds a crossing to the projection of $J\cap B$ onto $D$, and then reseals the arcs.  Each subsequent round of processive recombination cuts the same arcs, adds an identical crossing between the arcs, and reseals the arcs in exactly the same manner as the first round.  In particular, after  $|k|$ rounds of processive recombination the arcs of $J \cap B$ are replaced by a row of $k$ identical half-twists as illustrated in Figure \ref{figure:assumption3serine} with $n=k+1$ if $k>0$ and $n=k-1$ if $k<0$.

\item A tyrosine recombinase replaces the arcs of $J \cap B$ with arcs which have a projection onto $D$ with at most $2$ crossings, as illustrated in Figure \ref{figure:assumption3tyrosine} with $|k|\leq 2$. 
\end{enumerate}
\end{assumption}

\begin{figure}[htpb]
\includegraphics[height=1.2in]{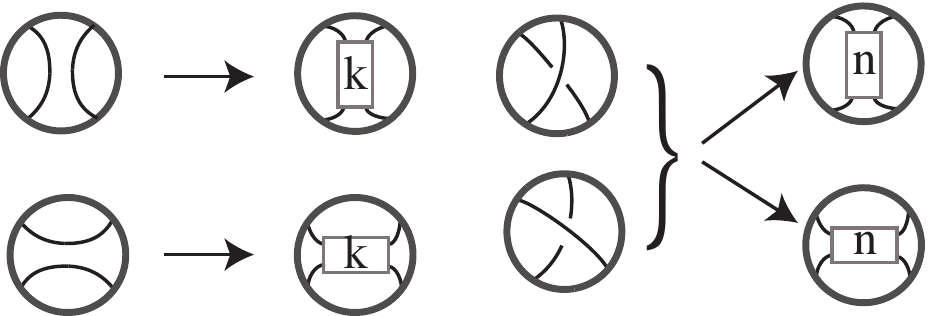}
\caption{The effect on $J\cap B$ of $|k|$ rounds of processive recombination mediated by a serine recombinase.}  \label{figure:assumption3serine}
\end{figure}

\begin{figure}[htpb] 
\includegraphics[height=2.3in]{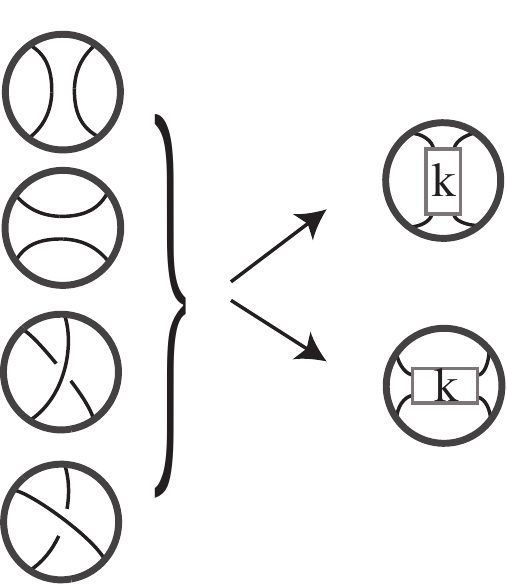}
\caption{The effect on $J\cap B$ of recombination mediated by a tyrosine recombinase. } \label{figure:assumption3tyrosine}
\end{figure}

\section{Lemmas derived from our assumptions} \label{section:thm}

For the remainder of the paper we will assume that our substrate and enzyme satisfy Assumptions 1, 2, and 3.  To characterize the possible knotted and linked products of recombination, we first use Assumptions 1 and 2 to reduce our analysis to a small number of possible configurations of the substrate-enzyme complex prior to recombination.  Then we apply Assumption 3 to alter the strands within the enzyme ball and identify the resulting products.   

More specifically, it follows from Assumption~1 that there is a disk $D$ properly embedded in the enzyme ball $B$ such that a projection of $J\cap B$ onto $D$ has one of the forms illustrated in Figure~\ref{figure:assumption1}.  Then it follows from Assumption~2 that after an ambient isotopy pointwise fixing $B$, the substrate $J$ bounds a surface $R$ represented by a decorated surface $S$ contained in a sphere $P$ satisfying the three conditions of Assumption~2.  Thus prior to recombination, the substrate-enzyme complex can be understood from the planar diagram $S\cup D\subseteq P$.  Finally, by Assumption 3, the action of the enzyme on the substrate $J$ can be seen as a change in the projection of the arcs of $J\cap B$ onto $D$ as illustrated in Figures~\ref{figure:assumption3serine} and \ref{figure:assumption3tyrosine}. Together these three assumptions allow us to characterize the knotted and linked products by focusing on the planar diagram $S\cup D\subseteq P$ prior to recombination, rather than having to work with the configuration of the substrate-enzyme complex $J\cup B$ in $\mathbb{R}^3$.

Before stating and proving our characterization theorem, we use the three assumptions to prove a number of lemmas that will simplify our analysis.

\bigskip

\begin{lemma}\label{lemma:forms}
The possible forms of $S \cap D$ are illustrated in Figure~\ref{Lemma1}.   Furthermore, if $D$ intersects more than one component of $\partial S$, then $S\cap D$ only has one component.  
\end{lemma}

\begin{figure}[h]
 \includegraphics[]{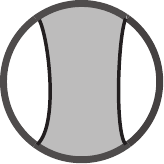}
\quad 
 \includegraphics[]{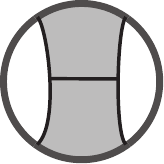}
 \quad 
  \includegraphics[]{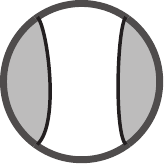}
\caption{Possible forms of $S \cap D$.} 
 \label{Lemma1}
\end{figure}

Note we will refer to $S\cap D$ in the first and second configurations as a {\it strip} and refer to $S\cap D$ in the third configuration as {\it two disks}.

\begin{proof}  By Assumption~1, $J\cap B$ consists of precisely two arcs whose four endpoints lie on the circle  $\partial D\subseteq \partial B$. Since $R$ is bounded by $J$, $\partial R\cap B$ consists of these same two arcs.  Also, it follows from Assumption~1 that $S$ has at most one decorating arc in $B$, and by condition (2) of Assumption~ 2 a neighborhood of any such arc must be disjoint from $\partial B$.  Thus regardless of whether $R\cap B$ is precisely $S\cap B$, or $R\cap B$ is obtained from $S\cap B$ by replacing one arc with a half-twisted band, $\partial S\cap B$ consists of two arcs with endpoints in $\partial B$.  

Now since $S$ and $\partial B$ are surfaces, $S\cap \partial B$ must contain two arcs between the endpoints of the arcs in $\partial S\cap B$.  Also, by condition (1) of Assumption~2, $S\subseteq P$ and $P\cap B=D$.  Hence $S\cap B\subseteq D$.  In addition, by Assumption~1, $D$ is properly embedded in $B$.  Hence $S\cap \partial B$ can have no circle components.  It follows that $S\cap \partial B$ consists solely of the two arcs joining the endpoints of  the two arcs of $\partial S\cap B$.  Thus the surface $S\cap B$ is bounded by two arcs in $\partial B$ and two arcs properly embedded in $B$.

Since, as we saw above, $S\cap B\subseteq D$, we must have $S\cap B\subseteq S\cap D$.  On the other hand, since $D\subseteq B$, we have $S\cap D\subseteq S\cap B$, and hence $S \cap  D = S \cap  B$.   Thus $S\cap D$ is bounded by two arcs in $\partial D$ and two arcs properly embedded in $D$.  If these four arcs together bound a single surface in $D$, then $S\cap D$ is a strip.  If the two pairs of arcs with shared endpoints each bound a surface in $D$, then $S\cap D$ is two disks.   Also, if there is a decorating arc in $S\cap D$, it is disjoint from $\partial D$.  It now follows that $S\cap D$ must have one of the forms illustrated Figure~\ref{Lemma1}.

Finally suppose that $D$ intersects more than one component of $\partial S$.  Since each component of $\partial S$ separates the sphere $P$, some component of $S \cap D$ intersects more than one component of $\partial S$.  Now it follows from Figure~\ref{Lemma1} that $S\cap D$ is a strip.\end{proof}

\medskip

\begin{lemma}  \label{lemma:nonugatory}
No decorating arc on $S - D$ separates $S \cup D$.  \end{lemma}

Figure~\ref{Lemma3} illustrates examples of decorating arcs which are excluded by Lemma~\ref{lemma:nonugatory}.

\begin{figure}[h]
\includegraphics[width=2.5in]{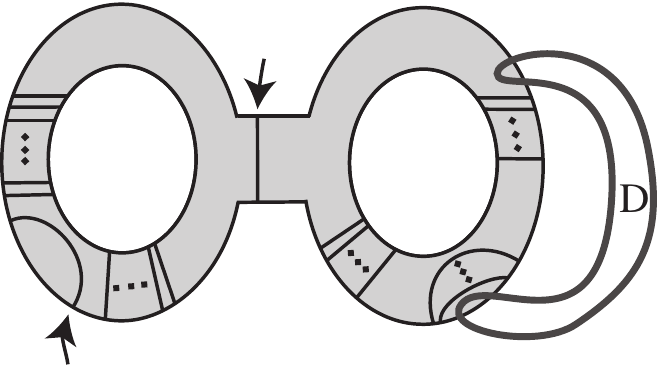}
\caption{Each of these decorating arcs separates $S \cup D$.} 
 \label{Lemma3}
\end{figure}

\begin{proof}  Suppose there exists a decorating arc $\alpha$ which separates $S\cup D$.  Then there is an ambient isotopy of $J$ pointwise fixing $B$ that removes the half-twist of $R$ represented by $\alpha$.  After this isotopy, $J$ bounds a surface that is represented by a decorated 2-holed disk with one less decorating arc which still satisfies the conditions of Assumption~2.  But this implies that the number of decorating arcs on our original surface $S$ was not minimal, contradicting condition~(3) of Assumption~2.\end{proof}
\medskip

\begin{lemma}\label{separatingpairs}  Any two decorating arcs on $S-D$ that together separate $S \cup D$ must belong to a single group of parallel arcs.\end{lemma}

Figure~\ref{Lemma4} illustrates a pair of decorating arcs which are excluded by Lemma~\ref{separatingpairs}.

\begin{figure}[h]
\includegraphics[width=2.3in]{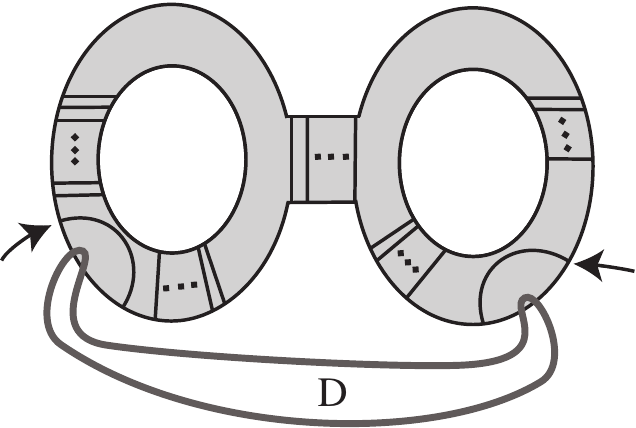}
\caption{This pair of decorating arcs together separates $S \cup D$.}  \label{Lemma4}
\end{figure}

\begin{proof} Suppose that $\alpha_1$ and $\alpha_2$ are arcs on $S - D$ belonging to distinct groups of parallel arcs such that $S \cup D - ( \alpha_1 \cup \alpha_2)$ has more than one component.   One of the components of $S \cup D - (N(\alpha_1) \cup N(\alpha_2))$ is disjoint from $D$.  Thus some component $Y$ of $R \cup B - ( N(\alpha_1) \cup N(\alpha_2))$ is disjoint from $B$.  Now we can ``flip $Y$ over'' while pointwise fixing $B$ and then simplify to remove the half-twist represented by $\alpha_1$ from $R$, while causing an additional half-twist represented by an arc parallel to $\alpha_2$ to be added to $R$ (see Figure~\ref{Lemma4Proof})

\begin{figure}[h]
\includegraphics[width=3.5in]{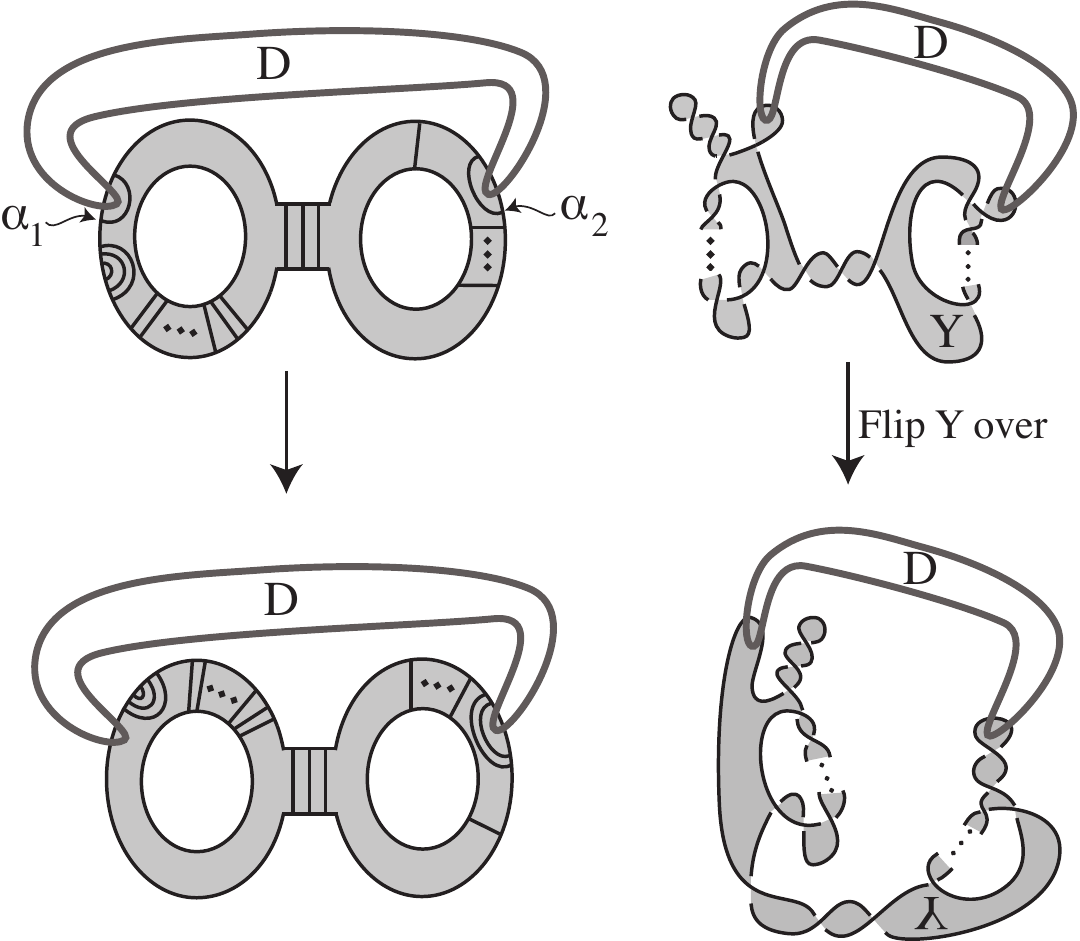}
\caption{We flip $Y$ over to remove $\alpha_1$ and add an arc parallel to $\alpha_2$.}  \label{Lemma4Proof}
\end{figure}

Continue this process, each time reducing the number of arcs on $S$ that are parallel to $\alpha_1$ by one, while increasing the number of arcs that are parallel to $\alpha_2$ by one.  In this way, we eventually eliminate the entire group of arcs parallel to $\alpha_1$.   However, this contradicts the minimality of the number of groups of parallel arcs required by condition (3) of Assumption~2.\end{proof}

\medskip

Before stating the next lemma, we need to introduce some additional terminology.   Given two distinct components of $\partial S$, we define the {\it algebraic twisting number} of the collection of all decorating arcs $S$ between these two components to be the number of these arcs representing positive half-twists of $R$ minus the number of these arcs representing negative half-twists of $R$.  Note that not all of the arcs in such a collection are necessarily parallel, since there may be arcs from one of the components to itself in the middle of the collection.

Observe that since $S$ is a 2-holed disk contained in the sphere $P$, the three components of $P-S$ are all disks.  However, for any point $x$ in $P-S$, the set $P-\{x\}$ is a plane and hence only two of the components of $P-\{x\}-S$ are disks.  In this case, we refer to the components of $\partial S$ which bound disks in $P-\{x\}-S$ as {\it inner components of $\partial S$ in $P-\{x\}$}.

\begin{lemma}  \label{s=0}
There is a point $x$ in $P-(S\cup D)$ such that the algebraic twisting number of the collection of arcs between the two inner components of $\partial S$ in the plane $P-\{x\}$ is  zero.
\end{lemma}

\begin{proof} Consider the collections of decorating arcs on $S$ that have endpoints on distinct boundary components of $\partial S$, ignoring any arcs with endpoints on the same component of $\partial S$.  Let $r_1$, $r_2$, and $r_3$ be the algebraic twisting numbers of these three collections (see Figure \ref{figure:pretzelS}).  If none of $r_1$, $r_2$, or $r_3$ is zero, then $J$ is a non-trivial pretzel link.  However, it follows from \cite{Montesinos, Waldhausen, Waldhausen2} that all non-trivial pretzel links are prime (see also \cite{Bonahon}).  But this is impossible since according to Assumption~2, $J$ is a connected sum of non-trivial torus links. Thus for some~$i$, we must have $r_i=0$.

 \begin{figure}[h]
\includegraphics[width=1.9in]{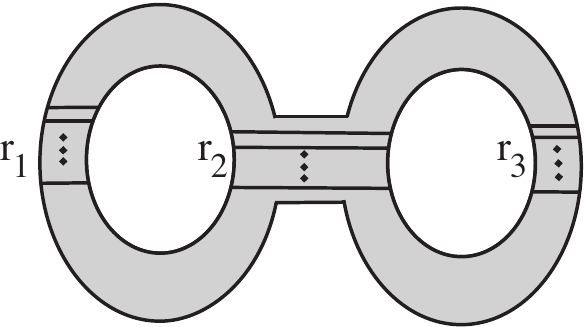}
\caption{$r_1$, $r_2$ and $r_3$ are the algebraic twisting numbers of the collections of parallel arcs between distinct components of $\partial S$.}  \label{figure:pretzelS}
\end{figure}

Let $X$ denote the collection of arcs with algebraic twisting number $r_i=0$. The arcs in $X$ all run between the same two components of $\partial S$.  So let $Z$ denote the component of $P-S$ which does not intersect any of the arcs in $X$.  If $Z$ were entirely contained in $D$, then the component of $\partial S$ bounding $Z$ would be entirely contained in $D$.  But this is impossible since $S\cap D$ has one of the forms illustrated in Figure~\ref{Lemma1}.  Thus we can choose a point $x\in Z-D$.  Now $x\in P-(S\cup D)$ and $X$ is the collection of arcs between the two inner components of $\partial S$ in the plane $P-\{x\}$.  Since $r_i=0$, we are done.   \end{proof}

\medskip 
 
From now on, we fix the point $x$ as given by Lemma~\ref{s=0} and draw all subsequent diagrams of $S\cup D$ in the plane $P-\{x\}$.  In particular, by Lemma~\ref{s=0}, $S$ looks like Figure~\ref{figure:pretzelS} with $r_2=0$.

\begin{lemma} \label{lemma:nocenterarcs}  Suppose that $S$ contains a decorating arc in the collection of arcs $X$ between the two inner components of $\partial S$ in the plane $P-\{x\}$.  Then $S$ has exactly two such arcs, one of which is contained in $D$, and the sign of the arc not in $D$ is opposite the sign of the arc in $D$.  
\end{lemma}

\begin{proof}  By Lemma~\ref{s=0}, the algebraic twisting number of the arcs in $X$ is zero.  Thus every arc in $X$ must be paired up with another arc in $X$ of opposite sign.  If neither arc in such a pair is contained in $D$, then there would be an ambient isotopy of $J$ pointwise fixing $B$ which would reduce the number of decorating arcs of $S$.  Since this contradicts the minimality of the number of arcs given in condition~(3) of Assumption~2, every arc in $X$ must be paired up with an arc of opposite sign in $D$.  However,  by Assumption~1, $S\cap D$ contains at most one decorating arc.  So $X$ contains exactly two arcs, one of which is contained in $D$, and the sign of the arc not in $D$ is opposite the sign of the arc in $D$. \end{proof}

\medskip

\begin{lemma}\label{configurations}  $S\cup D$ is depicted in the plane $P-\{x\}$ by one of the illustrations in Figure~\ref{figure:possibleSDcomplex}.

\end{lemma}

Note that when two cases are identical except the two inner components of $\partial S$ have been switched, we depict only one of the cases in Figure~\ref{figure:possibleSDcomplex}.  Also in the figure some of the groups of arcs that are illustrated might actually contain zero arcs.

\begin{figure}[htpb]
\begin{subfigure}[b]{0.45\textwidth}
\centering
\includegraphics[width=0.45\textwidth]{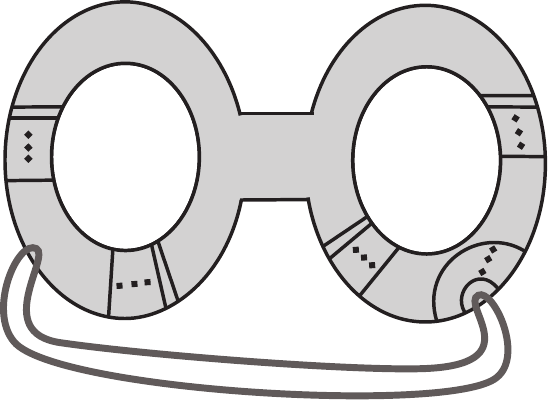}
 \caption{}  \label{figure:CasesFingersFar}
\end{subfigure}
\begin{subfigure}[b]{0.45\textwidth}
\centering
\includegraphics[width=0.45\textwidth]{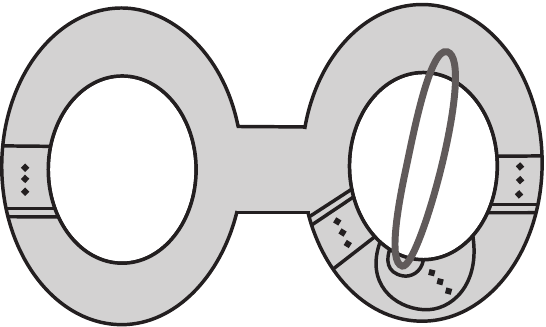}
\includegraphics[width=0.5\textwidth]{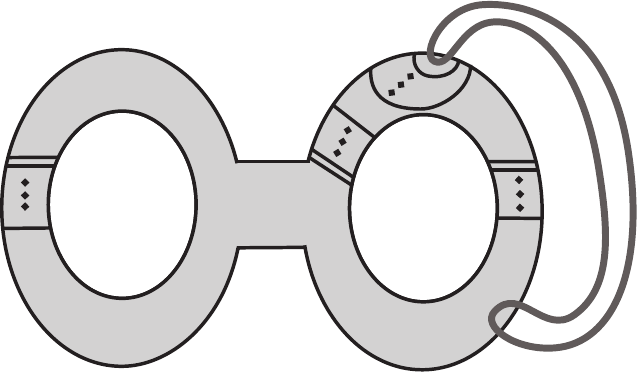}
\caption{}   \label{figure:CasesFingersNear}
\end{subfigure}

\begin{subfigure}[b]{\textwidth}
\centering
\includegraphics[width=0.2\textwidth]{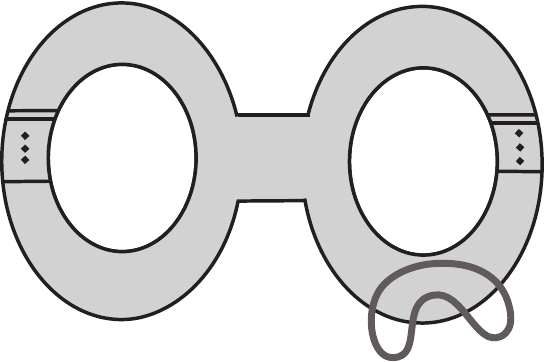}
\includegraphics[width=0.2\textwidth]{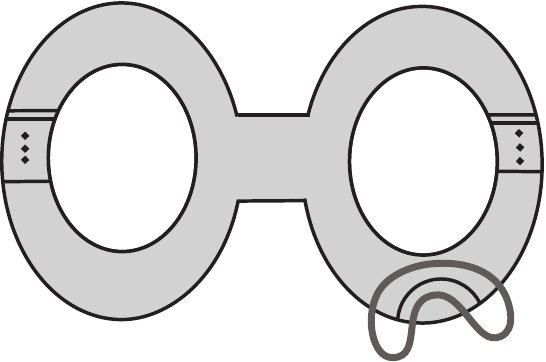}
\includegraphics[width=0.2\textwidth]{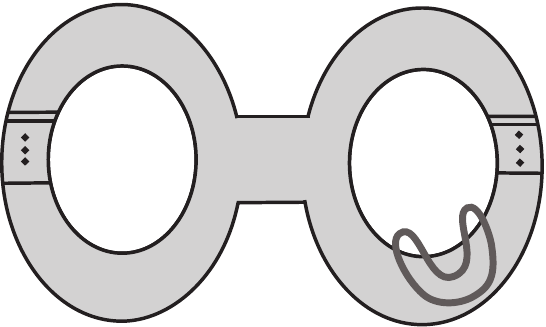}
\includegraphics[width=0.2\textwidth]{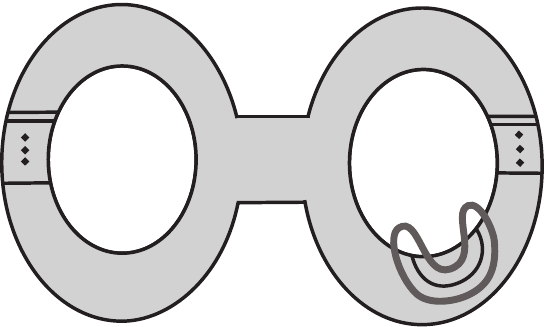}
\caption{}   \label{figure:CasesStripU}
\end{subfigure}
\begin{subfigure}[b]{0.45\textwidth}
\includegraphics[width=0.45\textwidth]{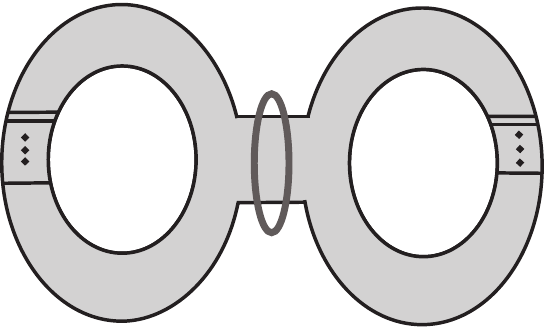}
\includegraphics[width=0.45\textwidth]{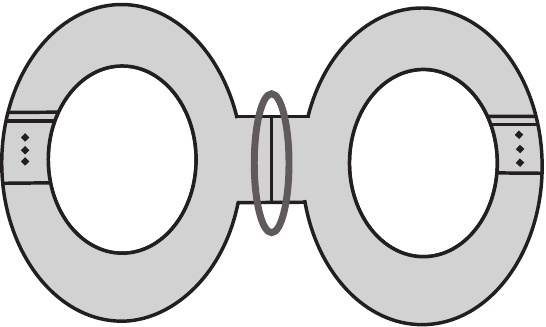}
\caption{}   \label{figure:CasesStripCenter}
\end{subfigure}
\begin{subfigure}[b]{0.45\textwidth}
\includegraphics[width=0.45\textwidth]{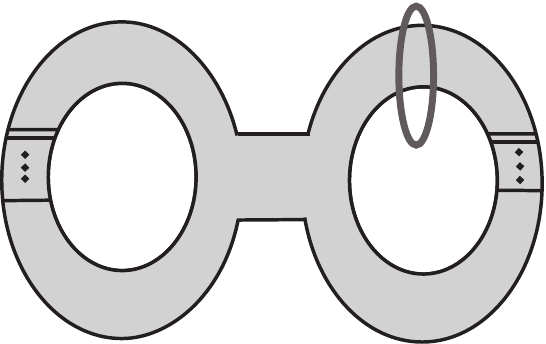}
\includegraphics[width=0.45\textwidth]{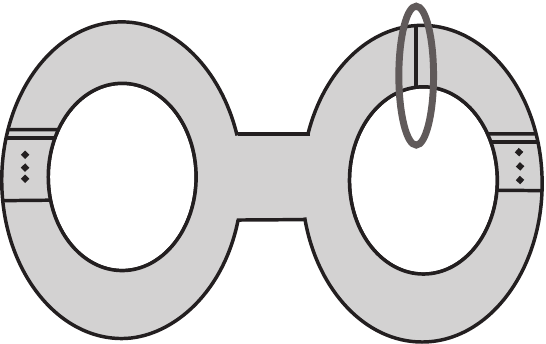}
\caption{}   \label{figure:CasesStripInOut}
\end{subfigure}

\begin{subfigure}[b]{0.45\textwidth}
\centering
\includegraphics[width=0.45\textwidth]{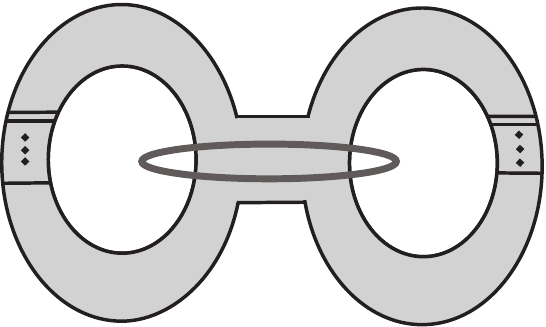}
\caption{} \label{figure:CasesStripCenterHorz}
\end{subfigure}
\begin{subfigure}[b]{0.45\textwidth}
\centering
\includegraphics[width=0.45\textwidth]{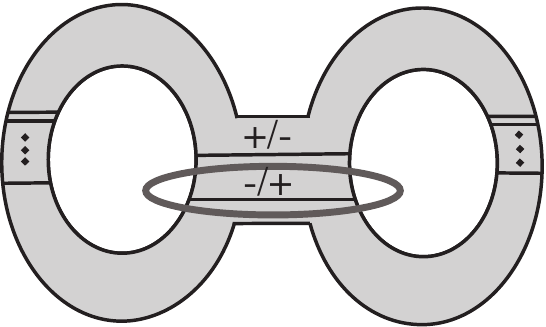}
\caption{} \label{figure:CasesStripCenterHorzArc}
\end{subfigure}

\caption{These are all possible forms of $S \cup D$ in the plane $P-\{x\}$.} \label{figure:possibleSDcomplex}
\end{figure}

\begin{proof}  We begin with some general observations that will apply to all cases.  First by Lemma~\ref{lemma:forms}, $S\cap D$ is either two disks or a strip containing at most one decorating arc.  Also, by Lemma~\ref{lemma:nonugatory}, no decorating arc on $S - D$ separates $S \cup D$.  
Furthermore, by Lemmas~\ref{s=0} and \ref{lemma:nocenterarcs}, if there are any arcs between the two inner components of $\partial S$ in the plane $P-\{x\}$, then there are precisely two such arcs with opposite signs and one of these arcs is contained in $D$.

Now we consider the cases, one at a time.  First suppose that $S \cap D$ is two disks.  Then it follows from the second part of Lemma~\ref{lemma:forms} that $D$ intersects only one component of $\partial S$.  Hence $S\cup D$ looks like the one of the illustrations to the left of an arrow in Figure~\ref{9ABSimplified}.  However, by Lemma~\ref{separatingpairs}, we can remove some arcs to get the illustrations to the right of the arrows in Figure~\ref{9ABSimplified}.  Hence $S\cup D$ has one of the forms depicted in Figure \ref{figure:CasesFingersFar} or \ref{figure:CasesFingersNear}.

 \begin{figure}[h]
\includegraphics[width=5.6in]{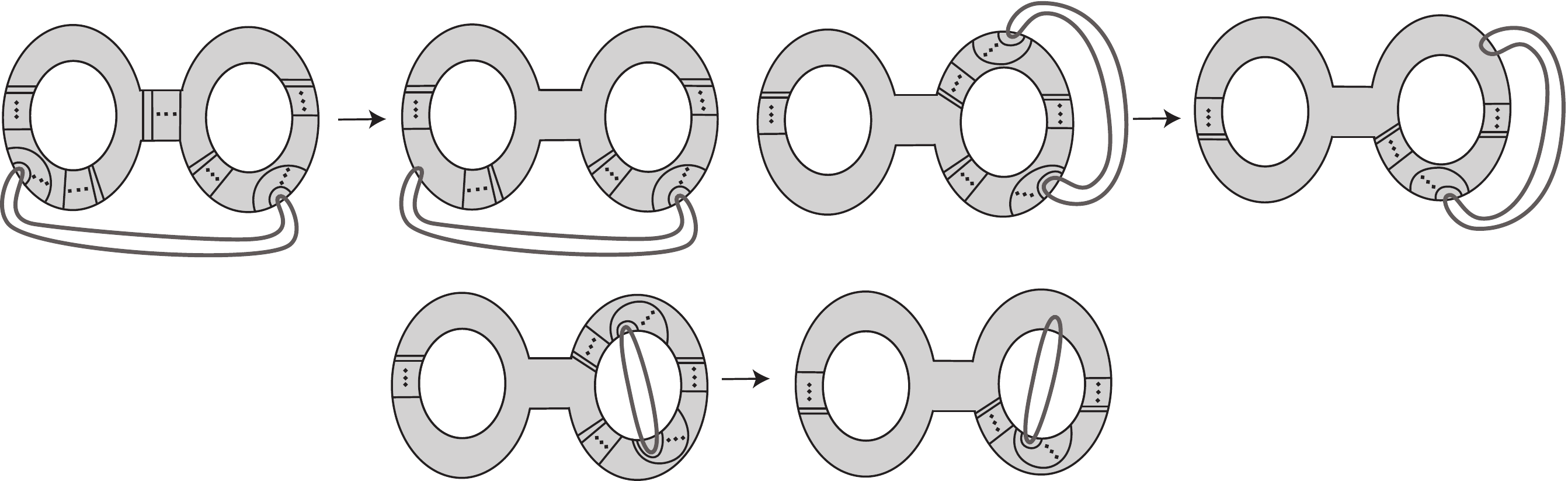}
\caption{Cases when $S \cap D$ is two disks.}  \label{9ABSimplified}
\end{figure}

Next suppose that $S\cap D$ is a strip intersecting only one component of $\partial S$.  If $B$ does not split apart the connected summands of $J=T(2,n)\#T(2,m)$, then $S\cup D$ looks like one of the illustrations in Figure~\ref{9CSimplified}.  Hence again by Lemma~\ref{separatingpairs}, the illustrations in the figure can be simplified to get those to the right of the arrows.  Hence $S\cup D$ has one of the forms illustrated in either Figure \ref{figure:CasesStripU} or Figure~\ref{figure:CasesStripCenter} according to whether or not $B$ splits apart the connected summands of $J=T(2,n)\#T(2,m)$.

 \begin{figure}[h]
\includegraphics[width=5.5in]{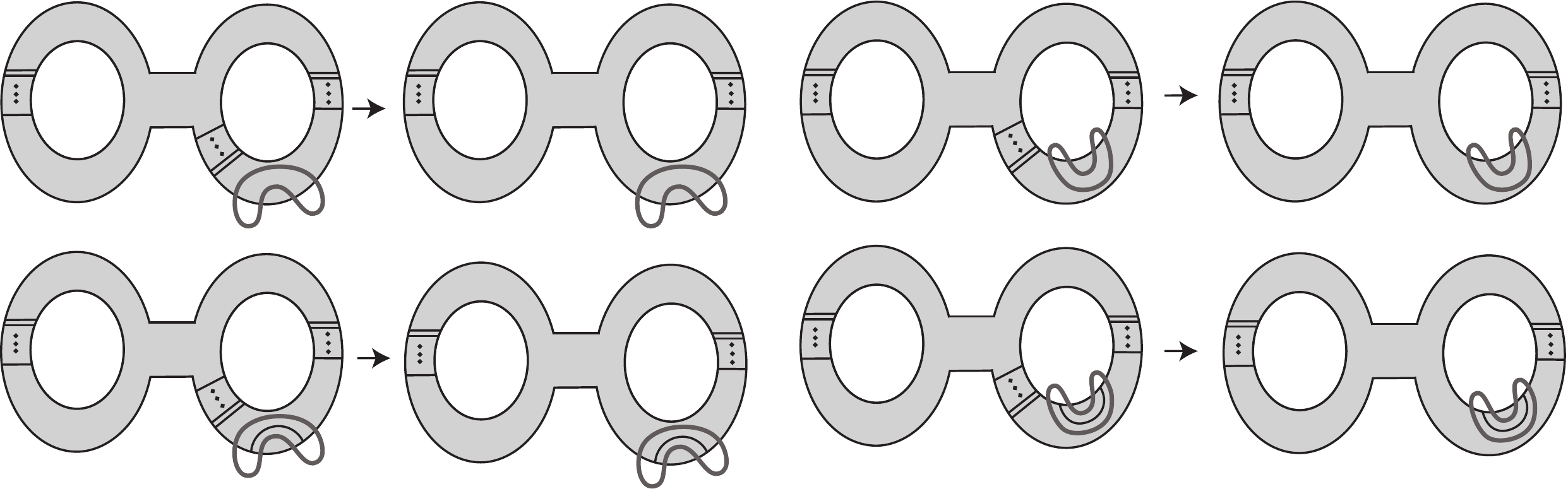}
\caption{Cases when $S\cap D$ is a strip intersecting only one component of $\partial S$ and $B$ does not split $J=T(2,n)\#T(2,m)$.}  \label{9CSimplified}
\end{figure}

Now suppose that $S\cap D$ is a strip intersecting an inner component and the outer component of $\partial S$ in the plane $P-\{x\}$.  Then $S\cup D$ looks like the one of the illustrations in Figure~\ref{9ESimplified}, which can again be simplified by Lemma~\ref{separatingpairs} to get the illustrations to the right of the arrows.   Thus $S\cup D$ has one of the forms illustrated in Figure~\ref{figure:CasesStripInOut}.

 \begin{figure}[h]
\includegraphics[width=5.5in]{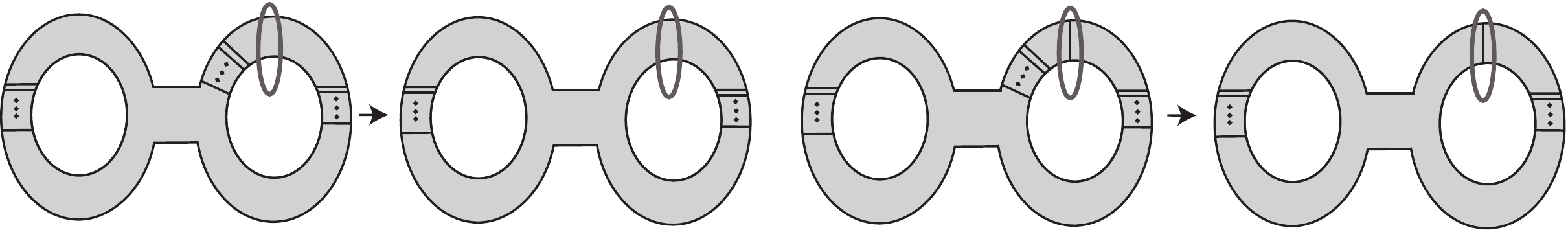}
\caption{Cases when $S\cap D$ is a strip intersecting an inner and  outer component of $\partial S$.}  \label{9ESimplified}
\end{figure}

Finally, suppose that $S\cap D$ is a strip intersecting both inner components of $\partial S$ in the plane $P-\{x\}$.  Then by Lemma~\ref{lemma:nocenterarcs}, $S\cup D$ has one of the forms illustrated in \ref{figure:CasesStripCenterHorz} or \ref{figure:CasesStripCenterHorzArc}.  
\end{proof}

\bigskip

\section{Classification Theorem}

\begin{thm} \label{thm}
Suppose that Assumptions 1, 2, and 3 hold for a particular serine or tyrosine substrate-enzyme complex with substrate $J=T(2,n)\#T(2,m)$.  Then every product of site-specific recombination is in one of the two families illustrated in Figure~\ref{Families}.\end{thm}

 \begin{figure}[h]
\includegraphics[width=3.5in]{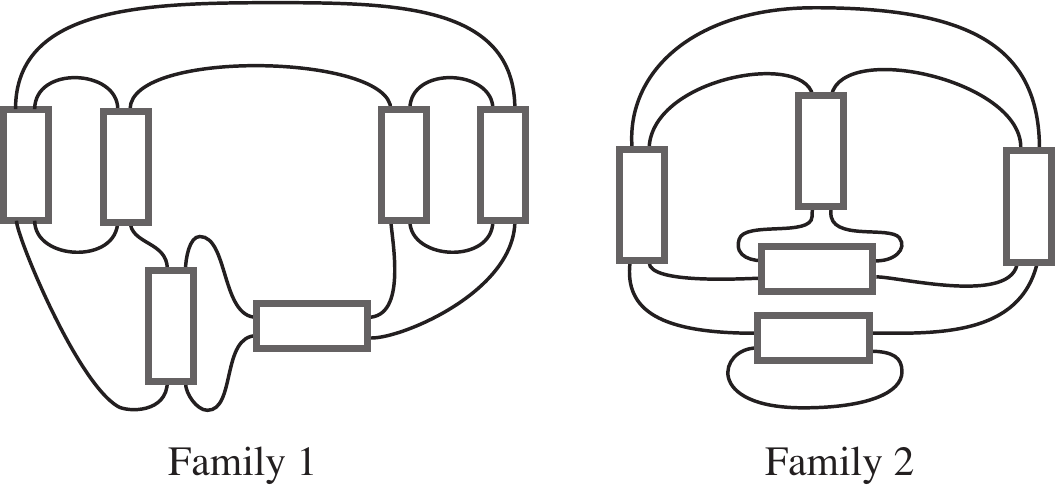}
\caption{Under the hypotheses of the theorem, every knotted or linked product is in one of these families.}  \label{Families}
\end{figure}

Note that in the proof we distinguish the products mediated by a serine recombinase from those mediated by a tyrosine recombinase.  

\begin{proof}  By Lemma~\ref{configurations}, we can assume that $S\cup D$ has one of the configurations in the plane $P-\{x\}$ illustrated in Figure~\ref{figure:possibleSDcomplex}.  Since these configurations describe all possibilities for the substrate-enzyme complex $J \cup B$, we can do our analysis entirely based on these illustrations. In what follows, we go through the configurations one at a time, applying Assumption~3 in each case to determine all possible knotted or linked products.  Note that if a rectangle contains no number or letter, then it is understood that there are no restrictions on the number of twists in the box.  However, crossings always go between the two short sides of the rectangle.

For the configuration of $S\cup D$ illustrated in Figure~\ref{figure:CasesFingersFar}, the possible products are illustrated in Figure~\ref{figure:ProductA}.  If the recombination is mediated by a serine recombinase, then we obtain links in Family~1 with no restrictions on $k$.  If the recombination is mediated by a tyrosine recombinase, then $|k| \leq 2$.   

 \begin{figure}[h]
\includegraphics[height=1.4in]{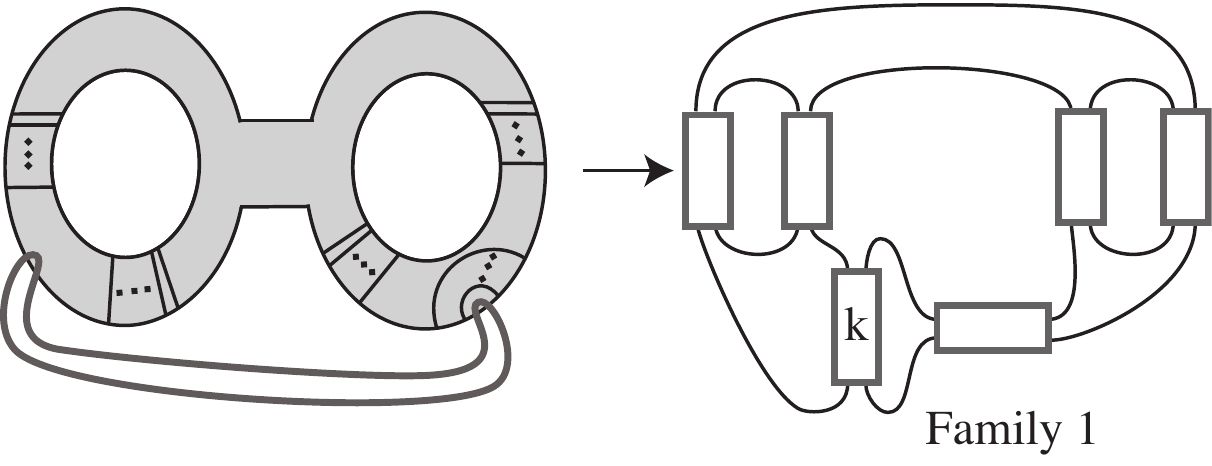}
\caption{Products when $S\cup D$ has the form of Figure~\ref{figure:CasesFingersFar}.} \label{figure:ProductA}
\end{figure}

For the configurations of $S\cup D$ illustrated in Figure~\ref{figure:CasesFingersNear}, the middle illustrations in Figure~\ref{figure:ProductB} describe the possible products.  However, after an isotopy we get the links on the far right.   If the recombination is mediated by a serine recombinase, then there are no restrictions on $k$.  If the recombination is mediated by a tyrosine recombinase, then $|k| \leq 2$.  Observe that in the first case we get links in Family 1, though the left and right sides of the illustration are reversed from what they are in Figure~\ref{Families}; and in the second case we get links in Family 2, though the top and bottom of the illustration are reversed from what they are in Figure~\ref{Families}. 
 \begin{figure}[h]
\includegraphics[height=3in]{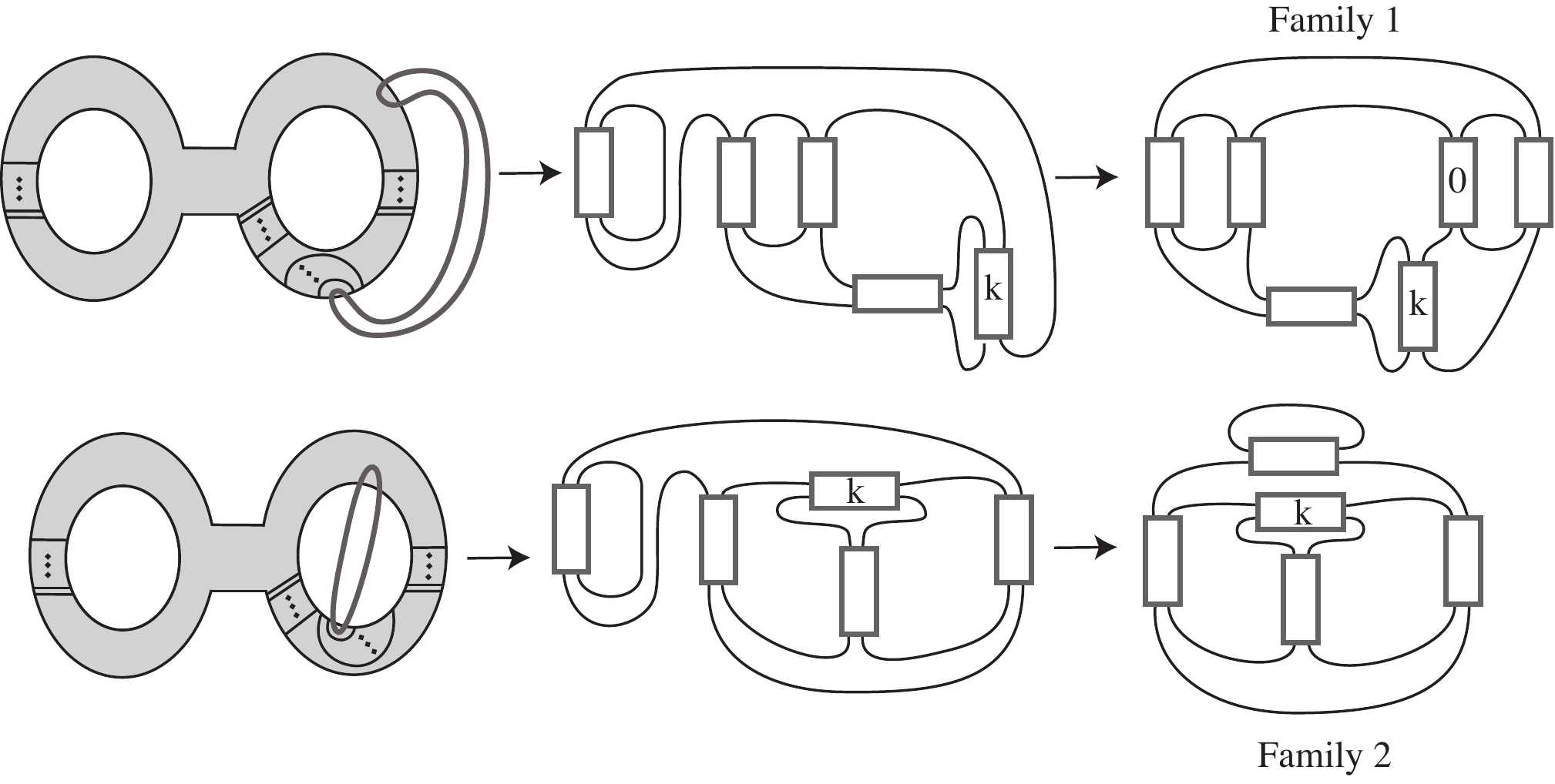}
\caption{Products when $S\cup D$ has the form of Figure~\ref{figure:CasesFingersNear}.  } \label{figure:ProductB}
\end{figure}

Next, consider the two configurations of $S\cup D$ in Figure~\ref{figure:CasesStripU}.  First suppose there is no decorating arc in $S \cap D$ (illustrated on the top left and right of Figure~\ref{figure:ProductC}).  Each round of processive recombination mediated by a serine recombinase inserts a nugatory crossing which can be removed by an isotopy.  Hence the products are the same as the substrate $J=T(2,n)\#T(2,m)$.  We can see the links $T(2,n)\#T(2,m)$ in the central illustration of Figure~\ref{figure:ProductC} by letting $k = \pm 1 $.    Now suppose there is a decorating arc in $S \cap D$ (illustrated on the bottom left and right of Figure~\ref{figure:ProductC}).  Then multiple rounds of processive recombination mediated by a serine recombinase replaces $J \cap B$ with a row of crossings. Since $S\cap D$ contains a decorating arc, we see from the illustration on the right in Figure~\ref{figure:assumption3serine} that this row can either be horizontal or vertical.  If the row of crossings is vertical, then the crossings are nugatory and we obtain the central illustration with $k = \pm 1$ as above.  If the row is horizontal, then we obtain a link again as in the central illustration of Figure~\ref{figure:ProductC}, and there are no restrictions on $k$.  Finally, suppose the recombination is mediated by a tyrosine recombinase.  Then $J \cap D$ is replaced by at most two consecutive crossings in either direction, and hence $|k| \leq 2$.

 \begin{figure}[htpb]
\includegraphics[height=2.2in]{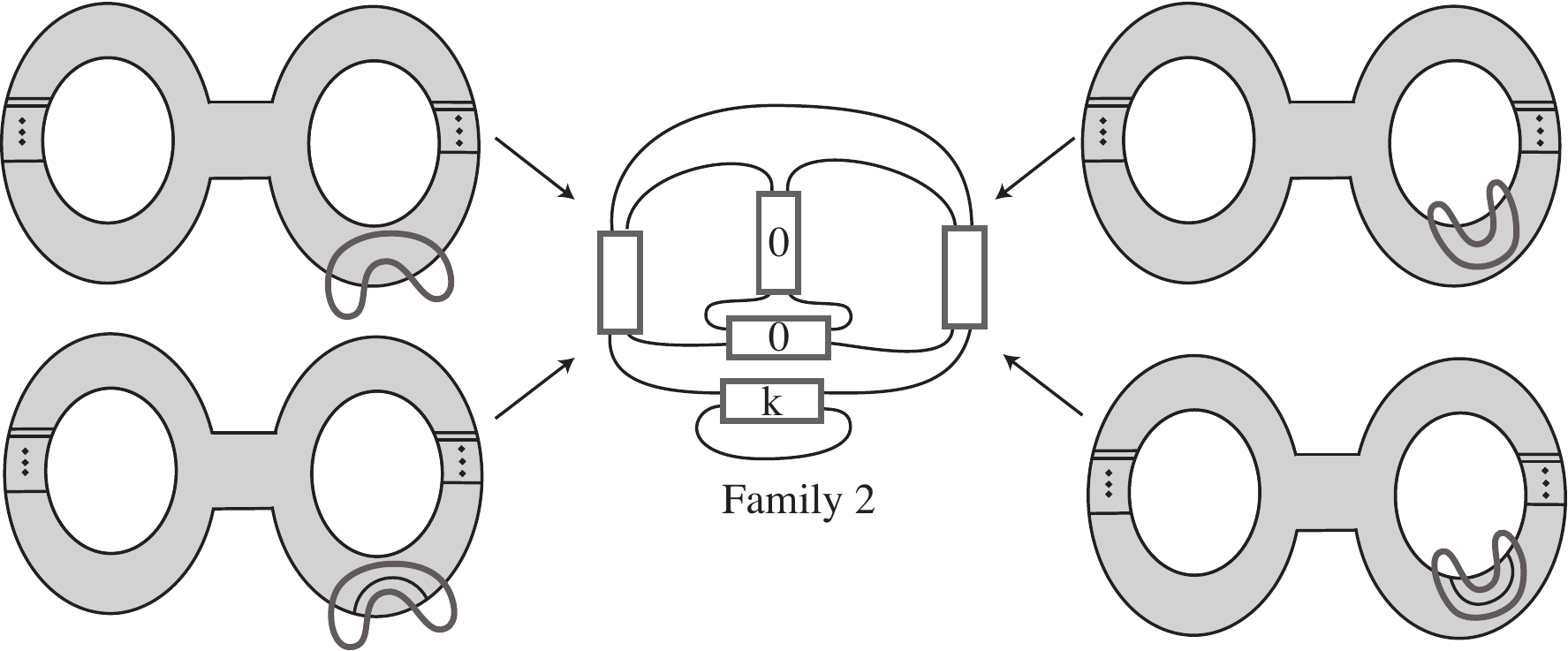}
\caption{Products when $S\cup D$ has the form of Figure~\ref{figure:CasesStripU}.}  \label{figure:ProductC}
\end{figure}

The analysis for the configurations of $S\cup D$ illustrated in Figure~\ref{figure:CasesStripCenter} is similar to Figure~\ref{figure:CasesStripU}.  When there is no decorating arc in $S \cap D$, each round of processive recombination mediated by a serine recombinase inserts a nugatory crossing which can be removed by an isotopy.  Hence the products again have the form $T(2,n)\#T(2,m)$, which we can see in the central illustration of Figure~\ref{figure:ProductD} by letting $k= \pm 1$.   When there is a decorating arc in $S \cap D$, then multiple rounds of processive recombination mediated by a serine recombinase replaces $J \cap B$ with either a row of nugatory crossings or a row of vertical crossings.  We can see the latter in the central illustration of Figure~\ref{figure:ProductD} with $|k| >1$.  
 If the recombination is mediated by a tyrosine recombinase, then $J \cap D$ is replaced by at most two consecutive crossings in either direction.  Hence in this case, $|k| \leq 2$.

 \begin{figure}[htpb]
\includegraphics[height=2.3in]{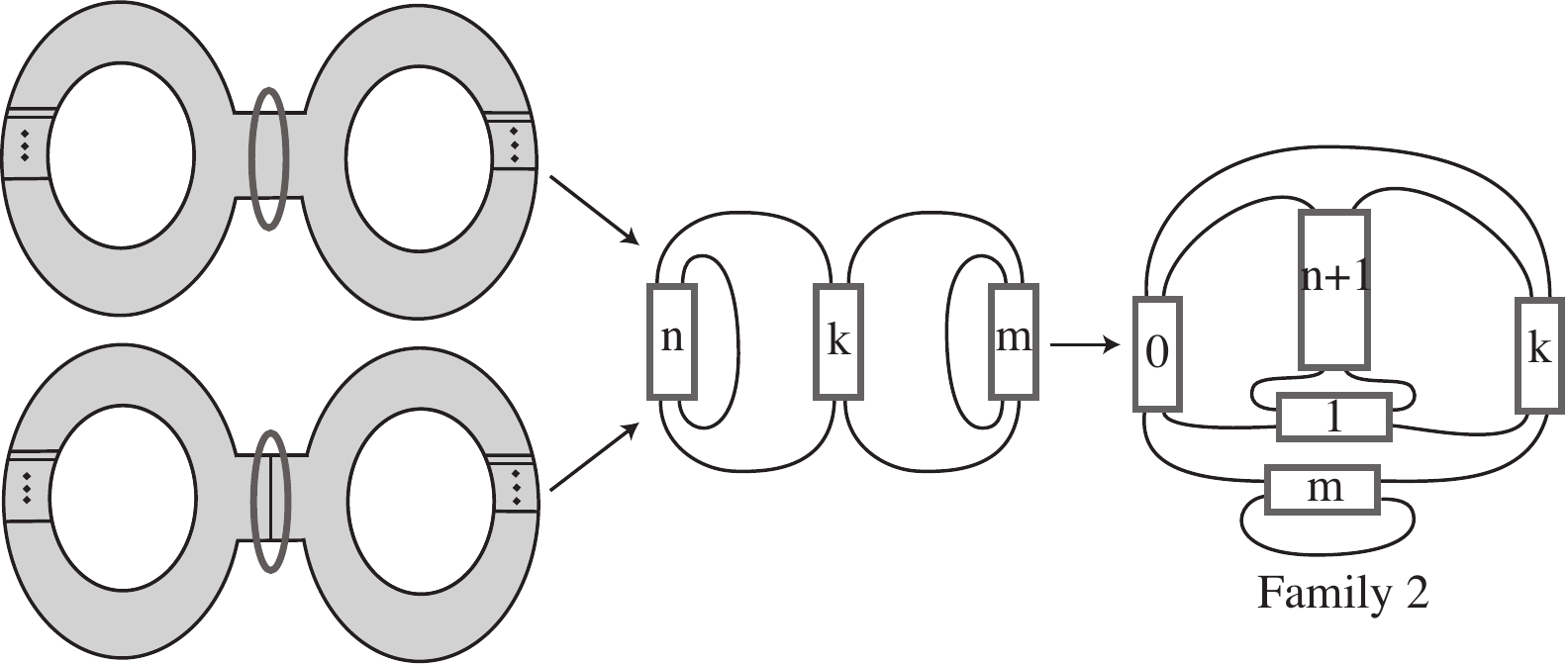}
\caption{Products when $S\cup D$ has the form of Figure~\ref{figure:CasesStripCenter}. } \label{figure:ProductD}
\end{figure}

Note that in Figure \ref{figure:ProductD}, $n$ and $m$ have the same values as they did in the substrate $T(2,n)\#T(2,m)$.  Also observe that a single crossing of $+1$ in a horizontal box is the same as a single crossing of $-1$ if the crossing were in a vertical box.  Thus the row of $n+1$ vertical crossings together with the single $+1$ horizontal crossing is equivalent to a row of $n$ vertical crossings.

For the configurations of $S\cup D$ illustrated in Figure~\ref{figure:CasesStripInOut}, the illustrations in Figure~\ref{figure:ProductE} describe the possible products.  When there is no decorating arc in $S \cap D$, each round of processive recombination mediated by a serine recombinase inserts a row of crossings which can be combined with the row of $m$ crossings in $J - B$ to get a row of $p$ crossings.  Thus the product is of the form $T(2,n)\#T(2,p)$.  We can see this in the central illustration of Figure~\ref{figure:ProductE} by letting $k= -1$, since a $-1$ in the center box becomes a $+1$ when it's joined with the vertical box with $p-1$ crossings.  When there is a decorating arc in $S \cap D$, then there are no restrictions on the value of $k$.  If recombination is mediated by a tyrosine recombinase, then $|k| \leq 2$.

 \begin{figure}[htpb]
\includegraphics[height=2.3in]{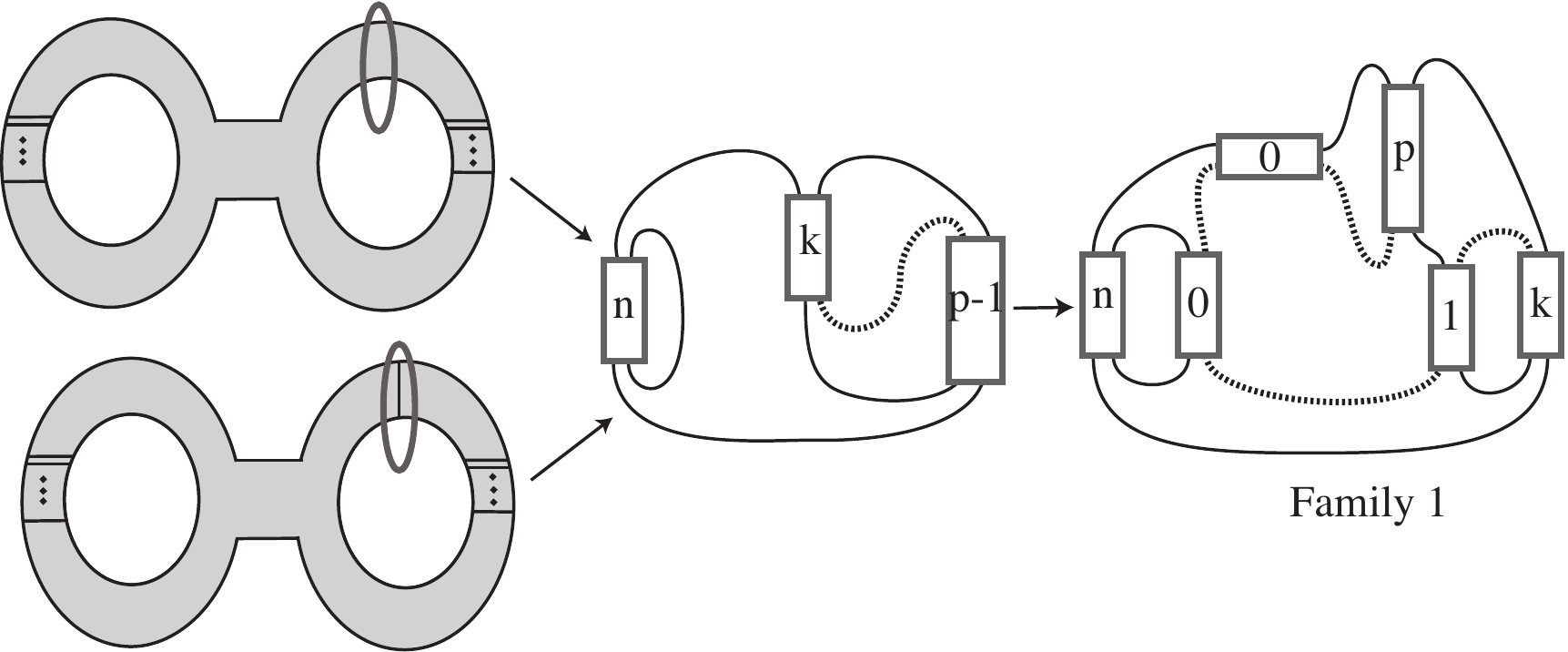}
\caption{Products when $S\cup D$ has the form of Figure~\ref{figure:CasesStripInOut}.  } \label{figure:ProductE}
\end{figure}

Note that in Figure \ref{figure:ProductE}, if $k\not=-1$, then $p-1=m$ from the substrate $T(2,n)\#T(2,m)$.  Also in the figure, we have highlighted some of the arcs with dotted lines to make it easier to see how to isotop the central illustration to obtain the illustration on the right.  This illustration is in Family 1, though it is rotated by $180^\circ$ relative to the illustration of Family 1 in Figure~\ref{Families}.  Finally, observe that in the illustration on the right the vertical box with $+1$ crossing becomes a $-1$ crossing when combined with the $p$ vertical crossings, thus giving us the box with $p-1$ vertical crossings in the center illustration.

For the configuration of $S\cup D$ in Figure~\ref{figure:CasesStripCenterHorz}, the illustration in Figure~\ref{figure:ProductF} describes the possible products.  For recombination mediated by a serine recombinase, the products are illustrated on the right with $k=\pm 1$ and no restrictions on the value of $j$.  For recombination mediated by a tyrosine recombinase, the products either have $j=0$ and $|k| \leq 2 $, or have $|j \pm 1| \leq 2$ and $k= \pm 1$ so that $|j-k|\leq 2$. 

 \begin{figure}[h]
\includegraphics[height=1.4in]{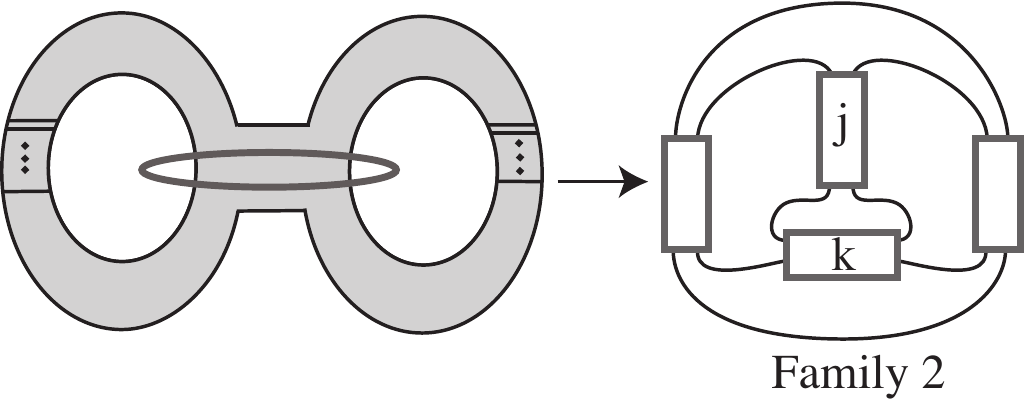}
\caption{Products when $S\cup D$ has the form of Figure~\ref{figure:CasesStripCenterHorz}. } \label{figure:ProductF}
\end{figure}

Finally, for the configuration of $S\cup D$ in Figure~\ref{figure:CasesStripCenterHorzArc}, the illustration in Figure~\ref{figure:ProductG} describes the possible products.  This is similar to the case illustrated in Figure~\ref{figure:ProductF}.  As above, for recombination mediated by a serine recombinase, the products are illustrated on the right with $k=\pm 1$  and no restrictions on the value of $j$.  For recombination mediated by a tyrosine recombinase, one box comes from the decorating arc outside of $D$ and hence is $\pm 1$. The other box is brought about by recombination mediated by a tyrosine recombinase and hence has absolute value no larger that $2$.  However, there is no restriction on which box is which.
 \end{proof}


 \begin{figure}[h]
\includegraphics[height=1.4in]{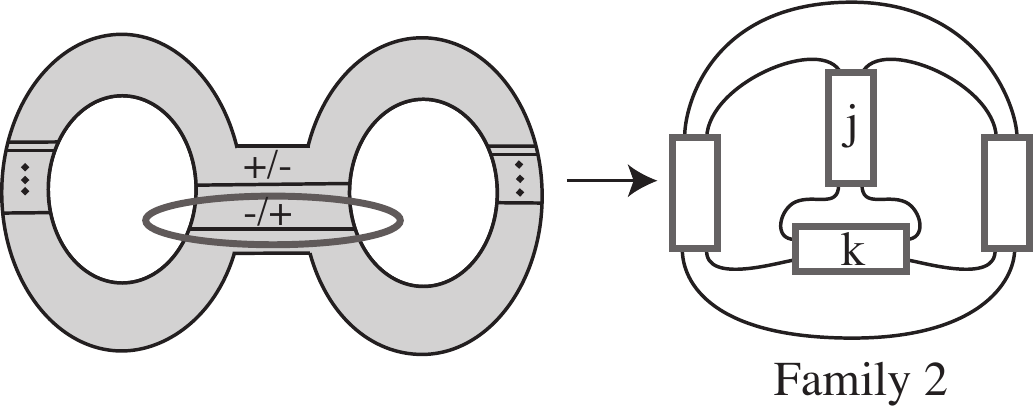}
\caption{Products when $S\cup D$ has the form of Figure~\ref{figure:CasesStripCenterHorzArc}. } \label{figure:ProductG}
\end{figure}

\medskip

\begin{rem}
For the configurations of $S\cup D$ illustrated in Figures~\ref{figure:CasesFingersNear}, \ref{figure:CasesStripU}, \ref{figure:CasesStripCenter}, and \ref{figure:CasesStripInOut}, the ball $B$ lies on one side of a sphere which separates the two connected summands of $J = T(2, m) \# T(2,n)$.  Buck and Flapan show in  \cite{BuckFlapan} that under three assumptions which are analogous to ours, if the substrate is a (possibly trivial) torus link $T(2,n)$, then the products must belong to a particular family depicted in Figure \ref{figure:BFknots}.  The products that we obtain in Figures~\ref{figure:CasesFingersNear}, \ref{figure:CasesStripU}, \ref{figure:CasesStripCenter}, and \ref{figure:CasesStripInOut} can be obtained as connected sums of a $T(2,m)$ with one of the links in the family in illustrated in Figure \ref{figure:BFknots}.  Thus our results in these cases are consistent with those of Buck and Flapan.

\end{rem}

\section*{Acknowledgements}
The authors would like to thank Dorothy Buck for suggesting this project.  The first author would also like to thank the Institute for Mathematics and its Applications for its hospitality during the Fall of 2013 when she was a long term visitor there.

\bibliographystyle{plain}
\bibliography{Biblio}

\end{document}